\documentclass[12pt,a4paper]{article}
\usepackage{amsmath,amscd,amsthm,amssymb,stmaryrd}
\usepackage{graphicx}
\usepackage{caption}
\usepackage{subcaption}

 \usepackage{graphicx}
\usepackage{cases}

  \def\r{\mathbb R}
\def\l{\mathbb L}
\def\h{\mathbb H}
\def\s{\mathbb S}
\def\e{\mathbb E}
\def\n{\mathbb N}
\newtheorem{theorem}{theorem}[section]
\newtheorem{corollary}[theorem]{Corollary}

\newtheorem{lemma}[theorem]{Lemma}
\newtheorem{proposition}[theorem]{Proposition}
\newtheorem{remark}[theorem]{Remark}
\newtheorem{example}[theorem]{Example}

\def\r{\mathbb{R}}
\def\h{\mathbb{H}}
 \def\n{\mathbf{n}}

\title{The Dirichlet problem of the constant mean curvature equation in Lorentz-Minkowski space and in Euclidean space}
\author{\textbf{Rafael L\'opez\footnote{This research has been partially supported by the grant no. MTM2017-89677-P, MINECO/AEI/FEDER, UE.}}
\\
Departamento de Geometr\'{\i}a y Topolog\'{\i}a\\ Instituto de Matem\'aticas (IEMath-GR)\\
 Universidad de Granada\\
 18071 Granada, Spain}

 \date{}
\begin{document}

  \maketitle
 
\abstract{We investigate the differences and similarities of the Dirichlet problem of the mean curvature equation   in the  Euclidean space and in the Lorentz-Minkowski space.   Although the solvability of the Dirichlet problem follows standards techniques of elliptic equations, we focus in showing how the spacelike condition in the Lorentz-Minkowski space allows to drop the hypothesis on the mean convexity which is required in  the Euclidean case. }



\section{Introduction}

In this paper we investigate the differences and similarities in the study of the solvability of the  Dirichlet problem for the constant mean curvature equation in the Euclidean space and in the Lorentz-Minkowski space.  Firstly we introduce the following notation. Let $\epsilon\in\{-1,1\}$. Denote by $\r^{n+1}_\epsilon$ the vector space  $\r^{n+1}$ equipped with the   metric  
$$\langle,\rangle=(dx_1)^2+(dx_2)^2+\ldots+(dx_n)^2+\epsilon (dx_{n+1})^2,$$
 where $(x_1,\ldots,x_{n+1})$ are the canonical coordinates of $\r^{n+1}$.  If $\epsilon=1$ (resp. $\epsilon=-1$), the space  is the Euclidean space $\e^{n+1}$ (resp. the Lorentz-Minkowski space $\l^{n+1}$).
 We consider the   Dirichlet problem for the constant mean curvature equation in $\r^{n+1}_\epsilon$. Let $\Omega\subset\r^n$ be a bounded  domain with smooth   boundary   $\partial\Omega$ and let $H$ be  a real number. The Dirichlet problem asks for existence and uniqueness of a function $u\in C^2(\Omega)\cap C^0(\partial\Omega)$ such that 
\begin{numcases}{}
 (1+\epsilon |Du|^2)\Delta u+\epsilon D_iuD_ju D_{ij}u=2H (1+\epsilon|Du|^2)^{3/2} & \mbox{in $\Omega$}\label{eq1}\\
 u=0 &\mbox{on $\partial\Omega$}\label{eq1-2}\\
  |D u|<1&  \mbox{in $\Omega$.\quad (if $\epsilon=-1$)}\label{eq1-3}
\end{numcases}
Here $D$ is the gradient operator, $D_i$ is the derivative with respect to the variable $x_i$ and the summation convention is used. A solution of \eqref{eq1}-\eqref{eq1-2}  describes a hypersurface with constant mean curvature $H$ in $\r_\epsilon^{n+1}$ whose boundary is contained in the hyperplane $x_{n+1}=0$. If   $\epsilon=-1$, the extra condition $|Du|<1$ in $\Omega$ means that the hypersurface is spacelike.  A hypersurface in $\e^{n+1}$ (resp. in $\l^{n+1}$) with zero mean curvature ($H=0$)  is called a minimal (resp. maximal)  hypersurface.

The example that shows  the differences of the theory of constant mean curvature hypersurfaces in both ambient spaces is  the Bernstein problem which we now formulate. Suppose that the domain $\Omega$ is $\r^n$. A graph on $\r^n$ is called an entire graph.   Let $H=0$.  The Bernstein problem asks if, besides linear functions, there are other entire solutions of  \eqref{eq1} with zero mean curvature. In the case $n=2$,    Bernstein  proved    that planes are the only entire minimal surfaces   (\cite{ber}). In arbitrary dimension, this result holds if $n\leq 7$.   A famous theorem of Bombieri,  De Giorgi and    Giusti asserts that there are other entire minimal graphs if $n\geq 8$   (\cite{bombieri}). In contrast, in $n$-dimensional Lorentz-Minkowski space, Cheng and Yau proved, extending previous works of Calabi, that spacelike  hyperplanes are the only entire maximal hypersurfaces (\cite{cy}).

The interest of the study of constant mean curvature  (cmc in short)  hypersurfaces has its origin in physics. In the Euclidean space $\e^3$, cmc surfaces are mathematical models    of the shape of a liquid in capillarity problems and of a  interface
that separates two medium of different physical   properties. In Lorentz-Minkowski $\l^{n+1}$, cmc spacelike hypersurfaces have been used in General Relativity to prove the positive mass theorem or analyze the space of solutions of Einstein equations (\cite{ch,mt}).

We review briefly the state of the art of the Dirichlet problem for the constant mean curvature equation in both spaces.  Assume that $u$ takes arbitrary continuous boundary values $u=\varphi$ on $\partial\Omega$. In the Euclidean  space and for the minimal case $H=0$, the Dirichlet problem \eqref{eq1} was solved for $n=2$ by Finn  \cite{fi} and in arbitrary dimension by Jenkins and Serrin \cite{js} proving that the mean convexity of the domain $\Omega$ yields a necessary and sufficient condition of the solvability of the Dirichlet problem  for all boundary values $\varphi$: a domain $\Omega$ is said to be mean convex if   the mean curvature $\kappa_{\partial\Omega}$ of $\partial\Omega$ with respect to the inner normal is non-negative. If $H\not=0$,   a stronger assumption is needed on $\Omega$ relating $H$ and $\kappa_{\partial\Omega}$ and  the answer appears in the seminal paper \cite{se}, where Serrin proved the following result. 

\begin{theorem}\label{t1}
The Dirichlet problem \eqref{eq1}   in the Euclidean  space  has a unique solution  for any boundary values $\varphi$  if and only if
\begin{equation}\label{c-se}
\kappa_{\partial\Omega}\geq \frac{n|H|}{n-1}\quad \mbox{ on $\partial\Omega$}.
\end{equation}
 \end{theorem}
 
It is expected that if we assume $\varphi=0$ on $\partial\Omega$, the assumption   \eqref{c-se}  may be relaxed. Indeed, if $\varphi=0$ and $n=2$,   the Dirichlet problem \eqref{eq1}-\eqref{eq1-2} has a unique solution if   $\kappa_{\partial\Omega}\geq |H|$ (\cite{lo10}): see other results in the Euclidean case. If we drop the convexity assumption of $\partial\Omega$, it is possible to derive existence results if one assumes smallness on  the domain $\Omega$ and certain uniform exterior sphere conditions:   see  \cite{be,gt,lo15}

The theory in $\l^{n+1}$ is shorter. The solvability of (\ref{eq1})-\eqref{eq1-3} with arbitrary boundary values was initially investigated in the maximal case $H=0$  assuming the mean convexity of $\partial\Omega$ (\cite{ban,fh}). However,  the groundbreaking result is due to Bartnik and Simon in 1982  where the counterpart   to Theorem \ref{t2} in $\l^{n+1}$ is surprisingly simple because there is not any assumption   on 
$\partial\Omega$   (\cite{bs}).

\begin{theorem}\label{t2}
 The Dirichlet problem  (\ref{eq1})-\eqref{eq1-3} in the Lorentz-Minkowski space    has a unique solution for any spacelike boundary values $\varphi$  if and only if   $\varphi$ has a spacelike extension to $\Omega$.
\end{theorem}

This result was later generalized in other Lorentzian manifolds:  \cite{ba,gh,gr,kl,th}. The method employed in the proof of Theorems \ref{t1} and \ref{t2} follows the Leray-Schauder fixed point theorem for elliptic equations because equation \eqref{eq1} is a quasilinear elliptic differential equation: if $\epsilon=-1$, this is assured   by the spacelike condition \eqref{eq1-3}. In order to apply standard methods in the solvability of the Dirichlet problem, we need to ensure  {\it a priori}  estimates of the height and the gradient for the prospective solutions. Throughout this paper, we refer to the reader \cite{gt} as a general guide.

 The purpose of this work is twofold. Firstly,     give an approach to the results  in Lorentz-Minkowski space comparing with the ones of Euclidean space and showing how the spacelike condition $|Du|<1$  makes completely different the method of obtaining  the  {\it a priori}  estimates. The second objective is to provide geometric proofs to derive these    estimates. For example,   Serrin used  the distance function to $\partial\Omega$ as a barrier for the desirable estimates (\cite{se}), and similarly  Flaherty in the solvability in the Lorentzian case when $H=0$ (\cite{fh}). This distance function is defined in $\Omega$ but   loses its geometric sense if we look    the graph of $u$ in $\e^3$ or $\l^3$. In our case, the   {\it a priori}  estimates will be obtained   by a comparison argument between the solutions of (\ref{eq1}) and known   cmc surfaces, such as, rotational surfaces.  In order to simplify the notation and arguments, we will consider the Dirichlet problem for the $2$-dimensional case, so we will work with surfaces in $\e^3$ and spacelike surfaces in $\l^3$. In such a case, the mean convexity of the curve $\partial\Omega$ is merely  the convexity of $\partial\Omega$.

This paper is organized as follows. After the  Preliminaries  section devoted to fix some definitions and notations, we derive the constant mean curvature equation in Section \ref{sec3}  obtaining some properties of the solutions showing differences in both ambient spaces.   Section \ref{sec4} describes the method of continuity to solve the Dirichlet problem \eqref{eq1}.  In  Section \ref{sec5}  we obtain the height estimates for solutions of \eqref{eq1} and we prove that   the boundary gradient estimates imply global (interior) gradient estimates. In Section \ref{sec7}, we analyze the solvability of the Dirichlet problem in the Euclidean case showing that a strong convexity hypothesis is necessary to solve the problem. Finally, in Section \ref{sec8} we solve the Dirichlet problem in Lorentz-Minkowski space for arbitrary domains and we show the role of the cmc rotational  surfaces in the solvability of the problem.

%%%%%%%%%%%%%%%%%%%%%%%%%%%%%%%%%%%%%%%5
\section{Preliminaries}\label{sec2}
%%%%%%%%%%%%%%%%%%%%%%%%%%%%%%%%

We need to recall some definitions in Lorentz-Minkowski space. In $\l^3$, the metric  $\langle,\rangle$ is   non-degenerate  of index $1$ and classifies  the vectors of $\r^3$ in three types: a vector $v\in\l^3$ is said to be
  spacelike (resp. timelike, lightlike) if $\langle v,v\rangle>0$ or $v=0$ (resp.  $\langle v,v\rangle<0$, $\langle v,v\rangle=0$ and $v\not=0$). 
The modulus of $v$ is $|v|=\sqrt{|\langle v,v\rangle|}$.  A vector subspace $U\subset\r^3$ is called  spacelike (resp. timelike, lightlike) if the induced metric on $U$ is positive definite (resp. non-degenerate of index $1$, degenerate and $U\not=\{0\}$). Any vector subspace belongs to one of the above three types. For $2$-dimensional subspaces, $U$ is spacelike (resp. timelike, lightlike) if its orthogonal subspace $U^\bot$ is timelike (resp. spacelike, lightlike). A curve or a surface immersed in $\l^3$   is said  to be  spacelike if the induced metric  is positive-definite. 

The spacelike property is a strong condition. For example, any spacelike surface $M$ is orientable. This is due because   a unit vector orthogonal to $M$ is timelike and in $\l^3$, the scalar product of any two timelike vectors is not zero. Thus, if we fix $e_3=(0,0,1)$, which is a timelike vector, it is possible to define a unit orthogonal vector field $N$ on $M$ so   $\langle N,e_3\rangle$ is negative (or positive) on $M$, determining a global orientation.    Another consequence  is that  there do not exist closed spacelike surfaces in $\l^3$, in particular, any compact spacelike surface has non-empty boundary.  Similarly, if a plane contains a closed spacelike curve,   the plane must be spacelike.

 Let $M$ be an orientable surface immersed in $\r_\epsilon^3$. In case $\epsilon=-1$, we also assume that the immersion is spacelike. Let $\nabla^0$ and $\nabla$ be  the Levi-Civita connections in $\r_\epsilon^3$ and $M$ respectively.  The Gauss formula is $\nabla_X^0 Y=\nabla_X Y+\epsilon \sigma(X,Y)$ for any two tangent vector fields $X$ and $Y$ on $M$, where $\sigma$ is the second fundamental form. The mean curvature $H$ of $M$ is defined as
\begin{equation}\label{eqh}
H=\frac12\mbox{trace}(\sigma).
\end{equation}
Let us choose  $N$   a unit normal vector field on $M$ with $\langle N,N\rangle=\epsilon$. Let $A=\nabla^0_N$ stand for  the Weingarten endomorphism with respect to $N$. Then the Gauss formula is  $\nabla_X^0 Y=\nabla_X Y+\epsilon \langle A(X),Y\rangle N$ and $A$ is a diagonalizable map. If $\kappa_1$ and $\kappa_2$ are the principal curvatures, we have 
$$H= \epsilon \frac12  \mbox{trace}(A)=\epsilon \frac12(\kappa_1+\kappa_2).$$

\begin{remark} In case of timelike surfaces of $\l^3$,   the mean curvature is defined as in (\ref{eqh}). However,  although $A$ is self-adjoint with respect to the  induced metric $\langle,\rangle$, this metric is Lorentzian and   it may occur that  $A$ is  not real diagonalizable. 
\end{remark}

\begin{example}\label{ex1}
\begin{enumerate}
\item Planes of $\e^3$ and spacelike planes of $\l^3$ have zero mean curvature.  
\item Round spheres $\s^2(r)$ in $\e^3$ and hyperbolic planes $\h^2(r)$ in $\l^3$ of radius $r>0$ can be described up to a rigid motion as 
$$\{p\in \l^3: \langle p,p\rangle=\epsilon r^2\}.$$
If $\epsilon=-1$, we also assume $\langle p,e_3\rangle<0$, where $e_3=(0,0,1)$. With respect to the Gauss map $N(p)=p/r$, the mean curvature is $H=-\epsilon/r$. 
\item Right circular cylinders of $\r^3_\epsilon$ have constant mean curvature. To be precise, let $a\in\r_\epsilon^3$ be a unit vector with $\langle a,a\rangle=1$ (in $\l^3$, the vector $a$ is spacelike). Up to a rigid motion,  the circular cylinder of axis $a$ and radius  $r>0$ is
$$C(r)=\{p\in \r_\epsilon^3: \langle p,p\rangle-\langle p,a\rangle^2=\epsilon r^2\}.$$   For the orientation $N(p)=(p-\langle p,a\rangle a)/r$,  the mean curvature is $H=-\epsilon /(2r)$. 
\item Let $u=u(x_1,x_2)$ be a smooth function defined in a open domain $\Omega\subset\r^2$ and let $M$ be the graph of $u$. Suppose that $M$ is endowed with the induced metric from $\r_\epsilon^3$. If $\epsilon=-1$, we also assume that $M$ is spacelike, that is,    $|Du|<1$ in $\Omega$.  The mean curvature  $H$ of $M$  satisfies
\begin{equation}\label{lo-media}
(1+\epsilon (D_2 u)^2)D_{11}u-2\epsilon D_1uD_2uD_{12}u+(1+\epsilon (D_1u)^2)D_{22}u=2H(1+\epsilon|Du|^2)^{3/2}
\end{equation}
with respect to the orientation
\begin{equation}\label{normal}
N =\frac{(-\epsilon D_1u,-\epsilon D_2u,1)}{\sqrt{1+\epsilon|D  u|^2}}=\frac{(-\epsilon D  u,1)}{\sqrt{1+\epsilon |D  u|^2}}\cdot
\end{equation}
Let us notice that (\ref{lo-media}) coincides with the equation  (\ref{eq1}).  
\end{enumerate}
\end{example}

%%%%%%%%%%%%%%%%%%%%%%%%%%%%%%%%%%%%%%%%%%%%%%%%%%%%%
\section{The constant mean curvature equation}\label{sec3}
%%%%%%%%%%%%%%%%%%%%%%%%%%%%%%%
In this section we will derive some properties on the solutions of the cmc equation \eqref{eq1}. The mean curvature equation \eqref{eq1} (or \eqref{lo-media})   can be expressed in the divergence form

 \begin{equation}\label{leq1}
 \mbox{ div }\Big(\dfrac{D  u}{\sqrt{1+\epsilon |D  u|^2}}\Big) =2H \quad \mbox{in $\Omega,$}\end{equation}
with the observation that if  $\epsilon=-1$, we   assume the spacelike condition  $|D  u|<1$ in $\Omega$. For instance, spheres   and  hyperbolic planes   of Example \ref{ex1} are  graphs of the functions 
$$u(x_1,x_2)=-\epsilon \sqrt{r^2-\epsilon(x_1^2+x_2^2)},	\quad\left\{\begin{array}{ll}x_1^2+x_2^2<r^2&\epsilon=1\\ (x_1,x_2)\in\r^2&\epsilon=-1.\end{array}\right.$$
For $\epsilon=1$, $x_3=u(x_1,x_2)$ is defined in a disc and describes a hemisphere in $\s^2(r)$, and for $\epsilon=-1$, $x_3=u(x_1,x_2)$ is the hyperbolic plane $\h^2(r)$. On the other hand, a cylinder $C(r)$ with axis $a=(0,1,0)$ and radius $r>0$ is the graph of the function 
$$u(x_1,x_2)=-\epsilon\sqrt{r^2-\epsilon x_1^2},	\quad\left\{\begin{array}{ll}|x_1|<r&\epsilon=1\\ (x_1,x_2)\in\r^2&\epsilon=-1.\end{array}\right.$$

Equation (\ref{leq1}) (with \eqref{eq1-3} if $\epsilon=-1$) is of quasilinear elliptic type, hence we can apply the machinery for these equations.  It is easily seen  that the difference of two solutions  of equation (\ref{eq1}) satisfies the maximum principle.  As a consequence, we give a statement of  the comparison principle in our context. We define   the operator      
\begin{equation}\label{eq4}
Q[u]=(1+\epsilon|Du|^2)\Delta u-\epsilon D_iuD_juD_{ij}u-2H(1+\epsilon |Du|^2)^{3/2}.
\end{equation} 
 The comparison principle asserts (\cite[Th. 10.1]{gt}).
 
 \begin{proposition}[Comparison principle] If $u,v\in   C^2(\overline{\Omega})$ satisfy $Q[u]\geq Q[v]$ in $\Omega$ and $u\leq v$ on $\partial\Omega$, then $u\leq v$ in $\Omega$.  If we replace $Q[u]\geq Q[v]$ by $Q[u]> Q[v]$, then   $u<v$ in $\Omega$. In particular, the solution of the Dirichlet problem, if it exists, is unique.
 \end{proposition}
 
An immediate consequence is the touching principle.
  
  \begin{proposition}[Touching principle]\label{p-tan} Let $M_1$ and $M_2$ be two   surfaces in $\r^3_\epsilon$ with the same constant mean curvature and   with possibly non-empty boundaries $\partial M_1$, $\partial M_2$.   If $M_1$  and $M_2$ have  a common tangent interior point and $M_1$ lies above $M_2$ around $p$, then $M_1$ and $M_2$ coincide at an open set around $p$. The same statement is also valid if $p$ is a common boundary point and the tangent lines to $\partial M_i$ coincide at $p$.
\end{proposition}

   A first difference  of the Dirichlet problem for the constant mean curvature equation \eqref{eq1}  is that in the Euclidean  space $\e^3$ the value $H$ is not arbitrary and  depends on the size of $\Omega$, whereas in $\l^3$ the value $H$ may be arbitrary. Indeed, from equation   (\ref{leq1}), the divergence theorem   yields
$$2|H|\mbox{area}(\Omega)=\left|\int_{\partial \Omega}\langle \frac{Du}{\sqrt{1+\epsilon |Du|^2}},\vec{n}\rangle\right|,$$
where $\vec{n}$ is the outward unit normal vector along $\partial\Omega$.
The idea is   to estimate  the right-hand side from above. If $\epsilon=1$, we have
$$2|H|\mbox{area}(\Omega)=\left|\int_{\partial \Omega}\langle \frac{Du}{\sqrt{1+|Du|^2}},\vec{n}\rangle\right|\leq  \int_{\partial \Omega}  \frac{|Du|}{\sqrt{1+|Du|^2}}  <\int_{\partial \Omega}1=\mbox{length}(\partial\Omega),$$

\begin{proposition}A necessary condition for the solvability of the Dirichlet problem \eqref{eq1}  in $\e^3$   is
\begin{equation}\label{uh}
|H|<\frac{\mbox{length}(\partial\Omega)}{2 \mbox{ area}(\Omega)}\cdot
\end{equation}
\end{proposition}

Let us notice that this upper bound for $H$ does not depend on the boundary values $\varphi$. In fact, there are explicit examples where all values between $0$ and the upper bound in (\ref{uh}) are attained. Indeed, let $\Omega$ be  a disc of radius $\rho$ and $\varphi=0$. Then the value of  $\mbox{length}(\partial\Omega)/(2 \mbox{ area}(\Omega))$ is $1/\rho$. On the other hand, for each $0<H<1/\rho$, take the spherical cap of radius $1/|H|$  
$$u(x_1,x_2)=-\sqrt{\frac{1}{H^2}-x_1^2-x_2^2},\quad x_1^2+x_2^2<\rho^2.$$
Then $u$ is a graph on $\Omega$ with constant mean curvature $H$ {\it for every $H$ going from $0$ until $1/\rho$}. The limit case $H=1/\rho$ corresponds with a hemisphere of radius $1/|H|$. 

The same computations in $\l^3$ do not provide the same conclusion because $|Du|/\sqrt{1-|Du|^2}$ may be arbitrarily large. So, for the hyperbolic planes  
\begin{equation}\label{hp}
u(x_1,x_2)=\sqrt{\frac{1}{H^2}+x_1^2+x_2^2}
\end{equation}
the value 
$$\frac{|Du|}{\sqrt{1-|Du|^2}}=|H|\sqrt{x_1^2+x_2^2}$$
is arbitrary large and the function $u$ is defined {\it in any domain of the plane $\r^2$ and for any $H$}. 

A second difference   is   the question of the existence of entire solutions of \eqref{eq1}  with non-zero mean curvature $H$: recall that the case $H=0$ (Bernstein problem) was discussed in the Introduction. In $\l^3$, the hyperbolic planes  (\ref{hp})   show that for any $H$, there are solutions \eqref{eq1} defined in the plane $\r^2$. Also the cylinders $u(x_1,x_2)=\sqrt{1/H^2+x_1^2}$  are other examples of entire solutions of \eqref{eq1}-\eqref{eq1-3}. However in the Euclidean  space, we have

\begin{proposition}\label{t-closure} Let $\Omega$ be a domain of $\r^2$. If $u$ is a solution of \eqref{eq1} with $H\not=0$ in $\e^3$, then $\Omega$ does not contain the closure of a disk of radius $1/|H|$.
\end{proposition}

\begin{proof} We proceed by contradiction. Assume that $D$ is an open disk of radius $1/|H|$ such that
$\overline{D}\subset\Omega$. Let $x$ be the center of $D$. Without loss of generality, we suppose that the sign of $H$ is positive: recall that the mean curvature is computed with respect to the orientation \eqref{normal}. Let $r=1/H$ and $\s^2(r)$ be a sphere of radius $r$ whose center lies on the straight-line through $x$ and perpendicular to the $(x_1,x_2)$-plane. Here, and in what follows, $\s^2(r)$ denotes a sphere of radius $r$   whose center may be changing. We orient $\s^2(r)$  by the inward orientation. With this choice of orientation, the mean curvature   is $H$ and the orthogonal projection of $\s^2(r)$ on $\r^2$ is $\overline{D}$.

Let $M$ be the graph of $u$.  Lift $\s^2(r)$ vertically upwards until $\s^2(r)$ is completely above $M$. Then, let us descend $\s^2(r)$ until
the first point $p$ of contact with $M$. Since $\overline{D}\subset \Omega$ and $M$ is a graph on $\Omega$, the contact point $p$ must be interior in both surfaces. By the touching principle, the surfaces $M$ and $\s^2(r)$ agree on an open set around $p$, hence  $M$ is included in a sphere of radius $1/H$: this is a contradiction because the orthogonal projection onto $\r^2$ would give $\Omega\subset\overline{D}$.
\end{proof}

%%%%% %%%%%%%%%%%%%%%%%%%%%%%
\section{The solvability techniques of the Dirichlet problem}\label{sec4}
 %%%%%%%%%%%%%%%%%%%%%%%
 
 In this section, we present the method for solving the Dirichlet problem (\ref{eq1})-\eqref{eq1-2}, which  holds in the Euclidean and Lorentzian contexts.   We establish the solvability  of the Dirichlet problem    by applying   the   method of continuity  (\cite[Sec. 17.2]{gt}). The matrix  of the coefficients of second order of (\ref{eq1}) is 
 $$\left(\begin{array}{ll}1+\epsilon (D_2u)^2&-\epsilon D_1u D_2 u\\ -\epsilon  D_1u D_2u&1+\epsilon (D_1u)^2\end{array}\right). $$
 The minimum and maximum eigenvalues of this matrix are $\lambda=1$ and $\Lambda=1+|Du|^2$ if $\epsilon=1$ and $\lambda=1-|Du|^2$ and $\Lambda=1$ if $\epsilon=-1$. Thus if $\epsilon=-1$, the equation \eqref{eq1} is uniformly elliptic   provided $|Du|<1$ uniformly in $\Omega$.  
 
 For $t\in [0,1]$, define  the family of Dirichlet  problems    
 $$ \left\{\begin{array}{ll}
Q_t[u]=0& \mbox{in $\Omega$}\\
 u = 0 &\mbox{on $\partial\Omega,$}\\
 |Du|<1& \mbox{on $\Omega$}\quad (\mbox{if $\epsilon=-1$})
 \end{array}\right.$$
 where 
   $$Q_t[u]= (1+\epsilon |Du|^2)\Delta u-\epsilon D_iuD_juD_{ij}u- 2tH(1+\epsilon |Du|^2)^{3/2}.$$
A solution $u$ of $Q_t[u]=0$ describe a surface with constant mean curvature $tH$.    As usual, let 
$$\mathcal{A}=\{t\in [0,1]: \mbox{there exists } u_t\in C^{2,\alpha}(\overline{\Omega}),   Q_t[u_t]=0, {u_t}_{|\partial\Omega}=0\}.$$ 
The existence of solutions of the Dirichlet problem \eqref{eq1}-\eqref{eq1-2}-\eqref{eq1-3} is established if   $1\in \mathcal{A}$. For this purpose, we prove that $\mathcal{A}$ is a non-empty open and closed subset of $[0,1]$. We analyze these three issues.

\begin{enumerate}
\item  \emph{The set  $\mathcal{A}$ is not empty}. This is  because $u=0$ solves the Dirichlet problem for $t=0$. 
\item \emph{  The set $\mathcal{A}$ is open in $[0,1]$}. Given $t_0\in\mathcal{A}$, we need to prove that there exists $\eta>0$ such that $(t_0-\eta,t_0+\eta)\cap [0,1]\subset\mathcal{A}$. Define the map $T(t,u)=Q_t[u]$ for $t\in\r$ and $u\in  C^{2,\alpha}(\overline{\Omega})$. Then $t_0\in\mathcal{A}$ if and only if $T(t_0,u_{t_0})=0$. If we show that the derivative  of $Q_t$ with respect to $u$, say $(DQ_t)_u$, at the point $u_{t_0}$ is an isomorphism,  the Implicit Function Theorem ensures the existence of an open set $\mathcal{V}\subset C^{2,\alpha}(\overline{\Omega})$, with $u_{t_0}\in \mathcal{V}$ and a $C^1$ function $\psi:(t_0-\eta,t_0+\eta)\rightarrow \mathcal{V}$ for some $\eta>0$, such that $\psi(t_0)=u_{t_0}>0$ and  $T(t,\psi(t))=0$ for all $t\in (t_0-\eta,t_0+\eta)$: this guarantees that $\mathcal{A}$ is an open  set of  $[0,1]$.

The map $(DQ_t)_u$ is one-to-one if   for any $f\in C^\alpha(\overline{\Omega})$, there is a unique solution $v\in C^{2,\alpha}(\overline{\Omega})$ of the linear equation $L[v]:=(DQ_t)_u(v)=f$ in $\Omega$ and $v=0$ on $\partial\Omega$. The computation of $L$ will be  done in Theorem \ref{pr7}, obtaining  
$$L[v]=(DQ_t)_uv=a_{ij}D_{ij}v+b_iD_iv,$$
where $a_{ij}=a_{ij}(Du)$ is symmetric, $b_i=b_i(Du,D^2u)$ and $L$ is a {\it linear} elliptic operator whose term for the function $v$ is zero. Therefore the existence and uniqueness is assured by standard theory (\cite[Th. 6.14]{gt}).

\item \emph{ The set $\mathcal{A}$ is closed in $[0,1]$}. Let $\{t_k\}\subset\mathcal{A}$ with $t_k\rightarrow t\in [0,1]$. For each $k\in\mathbb{N}$, there is $u_k\in C^{2,\alpha}(\overline{\Omega})$  such that $Q_{t_k}[u_k]=0$ in $\Omega$ and $u_k=0$ in $\partial\Omega$. Define the set
$$\mathcal{S}=\{u\in C^{2,\alpha}(\overline{\Omega}):  \mbox{there exists } t\in [0,1]\mbox{ such that }Q_{t}[u]=0 \mbox{ in }\Omega, u_{|\partial\Omega}=0\}.$$
Then $\{u_k\}\subset\mathcal{S}$. If we see that the set $\mathcal{S}$ is bounded in $C^{1,\beta}(\overline{\Omega})$ for some $\beta\in[0,\alpha]$, and since $a_{ij}=a_{ij}(Du)$ in (\ref{eq4}), the Schauder theory proves that $\mathcal{S}$ is bounded in $C^{2,\beta}(\overline{\Omega})$, in particular, $\mathcal{S}$ is precompact in $C^2(\overline{\Omega})$  (Th. 6.6 and Lem. 6.36 in \cite{gt}). Hence there is a subsequence $\{u_{k_l}\}\subset\{u_k\}$ converging to some $u\in C^2(\overline{\Omega})$ in $C^2(\overline{\Omega})$. Since $T:[0,1]\times C^2(\overline{\Omega})\rightarrow C^0(\overline{\Omega})$ is continuous, we obtain  $Q_t[u]=T(t,u)=\lim_{l\rightarrow\infty}T(t_{k_l},u_{k_l})=0$ in $\Omega$. Moreover, $u_{|\partial\Omega}=\lim_{l\rightarrow\infty} {u_{k_l}}_{|\partial\Omega}=0$ on $\partial\Omega$, so $u\in C^{2,\alpha}(\overline{\Omega})$ and consequently, $t\in \mathcal{A}$. The set $\mathcal{S}$ is bounded in   $C^{1,\beta}(\overline{\Omega})$ if it is bounded in $C^1(\Omega)$, where   the norm is defined by 
$$
\|u_t\|_{C^1(\Omega)}=\sup_\Omega |u_t|+\sup_\Omega|Du_t|.
$$
Usually, the  {\it a priori}  estimates for $|u|$ are called height estimates and gradient estimates for $|Du|$.

Definitively,   $\mathcal{A}$ is  closed in $[0,1]$ provided we find two constants $M$ and $C$ independent on $t\in\mathcal{A}$,  such that  
\begin{equation}\label{duu}
\sup_\Omega |u_t|<M,\quad \sup_\Omega|Du_t|<C.
\end{equation}
Here we make the   observation that whereas in the Euclidean  space, the constant $C$ can take an arbitrary value,  the spacelike condition in the Lorentz-Minkowski space implies that $C$ may be chosen to be $C=1$. However, during the above process of the method of continuity,  we require that $Q_t$ is uniformly elliptic, in particular, we have to ensure that $|Du|<<1$  in $\Omega$.  Definitively, in $\l^3$, the constant $C$ in \eqref{duu} has to satisfy the condition  $C<1$.

\begin{remark}In the Euclidean case,  the smoothness of the solution on $\partial\Omega$ is guaranteed if the graph close to the boundary point does not blow-up at infinity, that is, $|Du|\not\rightarrow \infty$. In the Lorentzian case,  we have to prevent the possibility that $|D  u|\rightarrow 1$ as we go to $\partial\Omega$. The existence of the constant $C$ shows that the surface  cannot `go null' in the terminology of Marsden and Tipler \cite[p. 124]{mt}.
\end{remark}

\end{enumerate}

%%%%%%%%%%%%%%%%%%%%%%%%%%%%%%%%%%%%%
\section{Height and gradient estimates}\label{sec5}
%%%%%%%%%%%%%%%%%%%%%%%%%%%%%%%%%%%%%%%

Consider the Dirichlet problem for the cmc equation and arbitrary boundary values
   \begin{equation}\label{leq2}
\left\{\begin{array}{ll}
 \mbox{ div }\Big(\dfrac{D  u}{\sqrt{1+\epsilon |D  u|^2}}\Big) =2H & \mbox{in $\Omega$}\\
 u=\varphi & \mbox{on $\partial\Omega,$}
\end{array}
\right.\end{equation}
where, in addition, if $\epsilon=-1$, we suppose $|Du|<1$ in $\Omega$. In this section we investigate the problem of finding  estimates of $|u|$ and $|Du|$ for a  solution $u$ of (\ref{leq2}) in terms of the initial conditions. In Theorems \ref{pr2}, \ref{t3} and \ref{t4} we will derive the estimates for $|u|$. For the gradient estimates, we will prove that the supremum of $|Du|$ in $\Omega$ is attained at some boundary point (Theorem \ref{pr7}). 

We begin with the height estimates. The main difference between both ambient spaces is that in $\e^3$  there exist estimates of $\sup_\Omega |u|$ depending only on $H$ and $\varphi$, whereas  in $\l^3$  the size of the domain $\Omega$   appears in these estimates, such as shows the   hyperbolic planes (\ref{hp}).

The height estimates for cmc graphs in the Euclidean space are obtained with the functions 
$$f(p)=\langle p,a\rangle,\quad g(p)=\langle N(p),a\rangle,\quad p\in M,$$
where $a$ is a fixed unit vector of $\r^3$ and $N$ is the Gauss map of $M$. Firstly   we need to compute the Beltrami-Laplacian $\Delta_M$ of the functions $f$ and $g$.  The following result holds for cmc surfaces in $\e^3$ and in $\l^3$ without to be necessarily graphs: we refer the reader to \cite{lop} for a proof. 

\begin{lemma} Let $M$ be an immersed surface in $\r_\epsilon^3$. Then   
\begin{equation}\label{lo-eq1}
\Delta_M\langle p,a\rangle=2H\langle N,a\rangle.
\end{equation}
If, in addition, the immersion has constant mean curvature, then
\begin{equation}\label{lo-eq2}
\Delta_M\langle N,a\rangle+\epsilon |\sigma|^2\langle N,a\rangle=0,
\end{equation}
where $|\sigma|$ is the norm of the second fundamental form.
\end{lemma}

Consider $u$ be a solution of \eqref{leq2} and let $M=\mbox{graph}(u)$.  If we take $a=e_3=(0,0,1)$, the functions $\langle p,e_3\rangle$ and  $\langle N,e_3\rangle$ inform about $u$ and $Du$ because
\begin{equation}\label{ifg}
\langle p,e_3\rangle=\epsilon u,\quad \langle N,e_3\rangle=\frac{\epsilon}{\sqrt{1+\epsilon|Du|^2}}.
\end{equation}
In particular, $\mbox{sign}(g)=\mbox{sign}(\epsilon)$. Suppose $H\geq 0$. Then $\Delta_M f\geq 0$ (resp. $\leq 0$) in $\e^3$ (resp. $\l^3$) and the maximum principle implies
$\langle p,e_3\rangle \leq \max_{\partial\Omega}\langle p,e_3\rangle$ in $\e^3$  (resp. $\langle p,e_3\rangle \geq \min_{\partial\Omega}\langle p,e_3\rangle$ in $\l^3$). Thus  $u\leq\max_{\partial\Omega}u$ in both ambient spaces. On the other hand
$$\Delta_M(Hf+\epsilon g)=(2H^2-|\sigma|^2)g\quad\left\{\begin{array}{ll}
\leq 0 &\epsilon=1\\
\geq 0 &\epsilon=-1.\end{array}\right.$$
Since $|\sigma|^2=\kappa_1^2+\kappa_2^2\geq 2H^2$, the maximum principle yields
$$ Hf+\epsilon g \left\{\begin{array}{ll}
\geq \min_{\partial\Omega}Hf+  g &\epsilon=1\\
\leq \max_{\partial\Omega} Hf- g &\epsilon=-1.\end{array}\right.$$

In case  $\epsilon=1$, we have 
$$Hu+\langle N,e_3\rangle\geq H\min_{\partial\Omega}u+\min_{\partial\Omega}\langle N,e_3\rangle\geq H\min_{\partial\Omega}u$$
because $\langle N,e_3\rangle\geq  0$. Since $\langle N,e_3\rangle\leq 1$, we deduce  $u\geq -1/H +\min_{\partial\Omega}\varphi $.

\begin{theorem} \label{pr2}
A solution $u$ of (\ref{leq2}) in the Euclidean  space satisfies
$$\min_{\partial\Omega}\varphi-\frac{1}{H}\leq u\leq \max_{\partial\Omega}\varphi, \quad \mbox{if $H>0$}$$
$$\min_{\partial\Omega}\varphi\leq u\leq   \max_{\partial\Omega}\varphi-\frac{1}{H}, \quad \mbox{if $H<0.$}$$
\end{theorem}

We analyze the same argument in $\l^3$.   The reverse Cauchy-Schwarz inequality for timelike vectors yields   $\langle N,e_3\rangle\leq -1$ (\cite{lo00}). Then the same computation gives
$$-Hu+\langle N,e_3\rangle\leq H\max_{\partial\Omega}(-u)+\max_{\partial\Omega}\langle N,e_3\rangle\leq -H\min_{\partial\Omega}u-1,$$
{\it but it is not possible to bound from below} because of the function $\langle N,e_3\rangle$.   This makes a key difference with the Euclidean case and concludes that  the argument done in the Euclidean  space is not valid in $\l^3$. If $H=0$, from \eqref{lo-eq1} we deduce:

\begin{corollary} In both ambient spaces, if $u$ is a solution of (\ref{leq2}) for $H=0$ then
$$\min_{\partial\Omega}\varphi\leq u\leq\max_{\partial\Omega}\varphi.$$
\end{corollary}

As expected, in the Lorentz-Minkowski space  there does not exist height estimates  depending only on $H$ and $\varphi$. An example is the following. For $r>0$ and $m>r$, let 
$u^m(x_1,x_2)=\sqrt{r^2+x_1^2+x_2^2}-m$ defined in the round disc $\Omega_{\sqrt{m^2-r^2}}=\{(x_1,x_2)\in\r^2:x_1^2+x_2^2<m^2-r^2\}$. The graph of $u^m$ is a piece of the hyperbolic plane $\h^2(r)$ which has been   displaced vertically downwards a distance  equal to $m$. Then $u^m$ is a solution of  (\ref{leq2}) in $\Omega_{\sqrt{m^2-r^2}}$ with $\varphi=0$ and the height on $u^m$, namely $|u^m|=m-r$,   goes to $\infty$ as $m\nearrow\infty$.

Motivated by these examples, we will deduce   height estimates   for a solution of (\ref{leq2}) in terms of the size of $\Omega$ (see \cite{lo16} for a height estimate in terms of the area of the surface). The estimates that we will deduce  are of two types: the first ones are given in terms of the diameter   of $\Omega$ and second ones depend on the width  of narrowest strip containing $\Omega$.

\begin{theorem}\label{t3}
 If $u$ be a solution of (\ref{leq2}) in   $\l^3$, then  
\begin{equation}\label{eqdia}\min_{\partial\Omega}\varphi-
\frac{1}{|H|}\left(\sqrt{1+\frac{\mbox{diam}(\Omega)^2 H^2}{4}}-1\right)\leq
u\leq \max_{\partial\Omega} \varphi+
\frac{1}{|H|}\left( \sqrt{1+\frac{\mbox{diam}(\Omega)^2 H^2}{4}}-1\right)
\end{equation}
and equality holds if and only if the graph of $u$ describes a hyperbolic cap. In the particular case   $\varphi=0$,   we have
$$\sup_\Omega|u|\leq \frac{1}{|H|}\left( \sqrt{1+\frac{\mbox{diam}(\Omega)^2H^2}{4}}-1\right).$$

\end{theorem}

\begin{proof} The inequalities  are obtained by comparing $M=\mbox{graph}(u)$  with hyperbolic caps with mean curvature $|H|$ coming from below and from above. There is no loss of generality in assuming that
 $\Omega$ is included in the closed disk $D_\rho$ of center the origin and radius $\rho=\mbox{diam}(\Omega)/2$. Consider the hyperbolic plane $\h^2(r)$ defined by the function   $u(x_1,x_2)= \sqrt{r^2+x_1^2+x_2^2}$,  where $r=1/|H|$. 
 
 Let us take $\h^2(r;s)$ the compact part  obtained  when we intersect  $\h^2(r)$  with the horizontal plane of equation $x_3=s$. Then    $\partial \h^2(r;s)$ is a circle of radius $\rho$, with   $s=\sqrt{\rho^2+r^2}$ and 
 $$\h^2(r;s)=\{(x_1,x_2,x_3)\in\h^2(r): x_3\leq s\}.$$
  Move vertically down   $\h^2(r;s)$    until to be disjoint from   $M$. Next move upwards $\h^2(r;s)$  until  that $\h^2(r;s)$ touches $M$ the first time. If the contact between both surfaces occurs at some common interior point, the comparison principle and then the touching principle implies that  $u$ describes part of the hyperbolic plane $\h^2(r;s)$. In such a case, the left inequality of (\ref{eqdia}) holds trivially. 
 
 In case that the first contact occurs between a point of $\h^2(r;s)$ with a boundary  point of $M$,   we can arrive until the value  $s=\min_{\partial\Omega}\varphi$, hence
 $$ \min_{\partial\Omega}\varphi-\sqrt{r^2+\rho^2}+\sqrt{r^2+x_1^2+x_2^2}\leq u\quad\mbox{in $\Omega$.}$$
 Evaluating at the origin,  
 $$ \min_{\partial\Omega}\varphi-\frac{1}{|H|}-\sqrt{\frac{1}{H^2}+\rho^2}\leq u\quad\mbox{in $\Omega$,}$$
 which coincides with the left inequality in (\ref{eqdia}) because $\rho=\mbox{diam}(\Omega)/2$.

 The right hand inequality in \eqref{eqdia} is proved with a similar argument by taking the hyperbolic planes $u(x_1,x_2)=-\sqrt{r^2+x_1^2+x_2^2}$.  
\end{proof}

A second height estimate can be deduced   by comparing   $u$ with   spacelike cylinders. We need to introduce the following notation. Given a bounded domain $A\subset\r^2$, consider the set $\mathcal{L}$ of all pairs  of parallel straight-lines $(L_1,L_2)$ in $\r^2$  such that $A$ is included in the planar strip   determined by $L_1$ and $L_2$. Set
$$\Theta(A)=\min\{\mbox{dist}(L_1,L_2): (L_1,L_2)\in \mathcal{L}\}.$$
Observe that the domain $A$ is included in a strip   of width $\Theta(\Omega)$ and this strip is the narrowest one among all strips containing $A$ in its interior. Notice also that $\Theta(A)\leq\delta(A)$.

\begin{theorem} \label{t4} If $u$ is a solution of (\ref{leq2}) in  $\l^3$, then 
\begin{equation}\label{es-lo-cy}\min_{\partial\Omega}\varphi-
\frac{1}{2|H|}\left(\sqrt{1+\Theta(\Omega)^2 H^2}-1\right)\leq u\leq \max_{\partial\Omega} \varphi+ \frac{1}{2|H|}\left(\sqrt{1+\Theta(\Omega)^2 H^2}-1\right).
\end{equation}
In the particular case   $\varphi=0$,   we have
$$\sup_\Omega|u| \leq \frac{1}{2|H|}\left(\sqrt{1+\Theta(\Omega)^2 H^2}-1\right).$$
\end{theorem}
Notice that the estimates  (\ref{es-lo-cy}) and \eqref{eqdia} are not comparable.
 
 \begin{proof}
 The argument is similar to the proof of Theorem \ref{t3} by replacing  the role of the hyperbolic planes by cylinders. 
 After a rigid motion if necessary, assume that $\Omega$ is included in the strip  $|x_1|<\Theta(\Omega)/2$. Consider the cylinder $C(r)$ 
 $$u(x_1, x_2)=\sqrt{r^2+x_1^2},$$
where $r=1/(2|H|)$. Consider the value $s$ such that the intersection of $C(r)$ with the plane of equation $x_3=s$ is formed by    two parallel straight-lines separated a distance equal to $\Theta(\Omega)$: this occurs when the value $s$ is 
$$s= \sqrt{r^2+\frac{\Theta(\Omega)^2}{4}}.$$
 Denote by $C(r;s)$ the part of $C(r)$ below the plane of equation $x_3=s$, which is a graph on a strip of width $\Theta(\Omega)$.   Let us move down the cylinders $C(r;s)$   until that do not intersect $M=\mbox{graph}(u)$. After, we move upwards $C(r;s)$   until the first touching point with $M$. If this point is a common interior point, then $M$ is included in the cylinder $C(r)$ and the left inequality in (\ref{es-lo-cy}) is trivially satisfied. If the point is not interior, we can arrive until the height $x_3=s$ where $s=\min_{\partial\Omega}\varphi$. Then 
 $$\min_{\partial\Omega}\varphi-\sqrt{r^2+\frac{\Theta(\Omega)^2}{4}}+\sqrt{r^2+x_1^2}\leq u \quad\mbox{in $\Omega.$}$$
 At the points $x_1=0$, we deduce
$$\min_{\partial\Omega}\varphi+r-\sqrt{r^2+\frac{\Theta(\Omega)^2}{4}}\leq u\quad\mbox{in $\Omega.$}$$
This inequality is just the left inequality in (\ref{es-lo-cy}). The right inequality in (\ref{es-lo-cy}) is proved  by comparing with the cylinders 
$u(x_1, x_2)=-\sqrt{r^2+x_1^2}$. 
 \end{proof}

 We finish this section investigating how to derive  the  {\it a priori}  estimates \eqref{duu} of $|Du|$ in $\Omega$. Recall that  we have to find a constant $C$ depending only on the initial data such that $|Du|\leq C$ in $\Omega$, with the observation that if $\epsilon=-1$, we require that $C<1$.  We will prove that it suffices to find this estimate only in boundary points. We   present two proofs of this result which hold in both ambient spaces.

\begin{theorem}\label{pr7} If $u$ is a solution of (\ref{leq2}), then
\begin{equation}\label{ingra}
\sup_{\Omega} |Du|=\max_{\partial\Omega}|Du|.
\end{equation}
\end{theorem}

\begin{proof}[Proof 1]
For each $i=1,2$, define the functions  $v^i=D_iu$. Differentiate (\ref{eq4}) with respect to the variable $x_k$, $k\in\{1,2\}$. After some computations, we obtain  
 \begin{equation}\label{eq3}
 \left((1+\epsilon |Du|^2)\delta_{ij}- \epsilon D_iuD_ju\right)D_{ij} v^k+2\left(\epsilon D_iu\Delta u+3H(1-|Du|^2) D_iu -\epsilon D_juD_{ij}u\right)D_iv^k=0.
 \end{equation}
  Hence $v^k$ satisfies a linear elliptic equation of type
  $$a_{ij}D_{ij}v^k+b_i D_iv^k=0,$$
  where $a_{ij}=a_{ij}(Du)$ and $b_i=b_i(Du,D^2u)$. By the maximum principle,  $|v^k|$ has not a maximum at some interior point. Consequently,   the maximum of $|Du|$ on the compact set $\overline{\Omega}$ is attained at some boundary point.
 \end{proof}
 
 \begin{proof}[Proof  2 ] Estimates of $|D  u|$  are obtained by means of the function $\langle N,e_3\rangle$  because 
 \eqref{ifg}.  From equation   \eqref{lo-eq2} 
 $$\Delta_M\langle N,e_3\rangle=-\epsilon|\sigma|^2\langle N,e_3\rangle =\frac{|\sigma|^2}{\sqrt{1+ \epsilon |Du|^2}}\leq 0,$$
 and   the maximum principle implies
 $$\inf_{\Omega}\langle N,e_3\rangle=\min_{\partial \Omega}\langle N,e_3\rangle.$$
Thus 
 $$\inf_{\Omega}\frac{\epsilon }{\sqrt{1+\epsilon |Du|^2}}=\min_{\partial \Omega}\frac{\epsilon }{\sqrt{1+\epsilon |Du|^2}},$$
 which is equivalent to (\ref{ingra}).
 \end{proof}
 
 To summarize, the problem of finding gradient estimates of $|Du|$ in $\Omega$ is passing to a problem of estimates along the boundary, exactly, finding a constant $C$ depending only on the initial data such that 
 \begin{equation}\label{gra}
 \max_{\partial\Omega}|Du|<C.
 \end{equation}

 In the proofs of the existence results in the following sections, the method to obtain the constant $C$ in (\ref{gra}) is by an argument of super and subsolutions and then we apply  the next result.

\begin{lemma}\label{le1} 
Let $x_0\in\partial\Omega$ be a boundary point. Suppose that there is a neighborhood $\mathcal{U}$ of $x_0$   and two functions $w^{+}, w^{-}\in C^2(\overline{\Omega\cap\mathcal{U}})$ such that 
$$
\begin{array}{lll}
&Q[w^+]\leq 0\leq Q[w^-]&\mbox{in $\Omega\cap U$}\\
&w^{-}\leq u\leq w^{+}&\mbox{in $\partial(\Omega\cap U)$}\\
&w^{-}(x_0)=u(x_0)=w^{+}(x_0)&\\
&|Dw^{-}|,|Dw^{+}|\leq C.&
\end{array}$$
Then $|Du|\leq C$. 
\end{lemma}
\begin{proof} The comparison principle yields 
$w^{-}\leq u\leq w^{+}$ in $\Omega\cap U$, concluding that $|Du|\leq\{|Dw^-|,|Dw^{+}|\}$. 
\end{proof}

%%%%%%%%%%%%%%%%%%%%%%%%%%%%
\section{The Dirichlet problem with zero boundary values: the Euclidean case}\label{sec7}
%%%%%%%%%%%%%%%%%%%%%%%%%%%%%%%%%%%%%
 
In this section we address  the Dirichlet problem \eqref{eq1} in the Euclidean  space. By Theorem \ref{pr2}, we know that   the value $H$ is not arbitrary. Without to assume convexity on $\partial\Omega$, there are results of existence assuming some smallness on the value $H$ and on the size of $\Omega$ (\cite{be,gt}. Thanks to this smallness on initial data, it is possible to  obtain height and boundary gradient estimate of the solution. If we assume   convexity, there are  different hypothesis that ensure the solvability of the Dirichlet problem and relate the size or the convexity of $\Omega$ with the value $H$ (\cite{lo10,lo13,lo14,lo15,lm2,ri,ri2}). 

 Theorem \ref{t1} solves the Dirichlet problem  in the Euclidean space for arbitrary boundary values. If we now suppose that $u=0$ on $\partial\Omega$, the hypothesis \eqref{c-se} can be weakened assuming    $\kappa_{\partial\Omega}\geq |H|$. We give two proofs of this result. The first one will be proved in arbitrary dimension and, although the idea appears generalized in other ambient spaces   (\cite{ad,lira,lo13-2,lm3}), as far as we know, in the literature there is not specifically a statement   in the Euclidean space. Here we follow   \cite{lm3}.    

\begin{theorem}\label{t5e} Let $H\not=0$. If the mean curvature of $\partial\Omega$ satisfies  $\kappa_{\partial\Omega} > |H|$, then the Dirichlet problem 
\begin{equation}\label{eq0}
\left\{\begin{array}{ll}
 \mbox{ div }\Big(\dfrac{D  u}{\sqrt{1+ |D  u|^2}}\Big) =nH & \mbox{in $\Omega$}\\
 u=0 & \mbox{on $\partial\Omega$}
\end{array}
\right.\end{equation}
in arbitrary dimension has a unique solution.
\end{theorem}

\begin{proof} Firstly, we observe that  the solutions $u_{t}$ of the method of continuity (Section \ref{sec4})  are ordered in decreasing sense according the parameter $t$. Indeed, if  $t_1<t_2$, then $Q_{t_1}[u_{t_1}]=0$ and 
$$Q_{t_1}[u_{t_2}]=(t_2-t_1) (1+|Du_{t_2}|^2)>0=Q_{t_1}[u_{t_1}].$$
  Since $u_{t_1}=0=u_{t_2}$ on $\partial\Omega$, the comparison principle yields $u_{t_2}<u_{t_1}$ in $\Omega$. Thus,    $u_1\leq u_t<0$ for all $t$, where for the value $t=1$, $u_1$ is the solution $u$ of (\ref{eq1}). By using   Lemma \ref{le1}, this implies   that it suffices to find  {\it a priori}  height and gradient   estimates for the prospective solution $u$ of \eqref{eq1}.

 If $u$ is a solution of \eqref{eq0}, then  $-u$ is a solution of \eqref{eq0} for the value $-H$. Thus, and without loss of generality, we suppose   $H>0$.  Let $M$ be the graph of $u$. By the height estimates of Theorem \ref{pr2}, we know $-1/H<u<0$ in $\Omega$.    This gives the  {\it a priori}  height estimates. According to Theorem \ref{pr7}, we need to find  {\it a priori}  boundary gradient estimates.  However, we will be able to find the gradient estimates on the domain $\Omega$.
 
 We use again the function $Hf+g$ as in Theorem \ref{pr2}. Since $\Delta_M(Hf+g)\leq 0$ and $u=0$ on $\partial\Omega$, the maximum principle ensures the existence of a boundary point $q\in\partial\Omega$ where $Hf+g$ attains its minimum, so 
\begin{equation}\label{eq5e-1}
H\langle p,e_3\rangle+\langle N,e_3\rangle\geq \min_{\partial\Omega}\langle N,e_3\rangle=\langle N(q),e_3\rangle.
\end{equation}
Furthermore, the maximum principle on the boundary implies 
$$H\langle \nu(q),e_3\rangle+\langle dN_q\nu,e_3\rangle\geq 0,$$
where $\nu$ is the inward unit conormal vector along $\partial\Omega$. If $\sigma$ is the second fundamental form, this inequality can be written as 
$$\left(H-\sigma(\nu(q),\nu(q))\right)\langle\nu(q),e_3\rangle\geq 0.$$
Since $u<0$ in $\Omega$, the boundary condition $u=0$ on $\partial\Omega$ yields $\langle\nu(q),e_3\rangle<0$, hence $H-\sigma(\nu(q),\nu(q))\leq 0$. 
If $\{v_1,\ldots,v_{n-1}\}$ is a orthonormal basis of the tangent space to $\partial\Omega$ at the point $q$, the above inequality implies
\begin{equation}\label{eq5e-2}
\sum_{i=1}^{n-1}\sigma(v_i,v_i)=nH-\sigma(\nu(q),\nu(q))\leq (n-1)H.
\end{equation}
Denote by $\nabla^{\partial\Omega}$ and $\sigma^{\partial\Omega}$ the Levi-Civita connection  and second fundamental form of $\partial\Omega$ as submanifold of $\Omega$, respectively. Let $\eta$ be the unit normal vector field of $\partial\Omega$ in $\Omega$. The Gauss formula gives
$$ {\nabla}^{0}_{v_i}v_i=\nabla_{v_i}v_i+\sigma(v_i,v_i)N(q)=\nabla^{\partial\Omega}_{v_i}v_i-\sigma^{\partial\Omega}(v_i,v_i)\eta(q)+\sigma(v_i,v_i)N(q).$$
Then $\sigma(v_i,v_i)=\sigma^{\partial\Omega}(v_i,v_i)\langle N(q),\eta(q)\rangle$. From \eqref{eq5e-2},
$$\langle N(q),\eta(q)\rangle\sum_{i=1}^{n-1}\sigma^{\partial\Omega}(v_i,v_i)\leq (n-1)H.$$
Since $\sum_{i=1}^{n-1}\sigma^{\partial\Omega}(v_i,v_i)=(n-1)\kappa_{\partial\Omega}$, we have 
$$\langle N(q),\eta(q)\rangle \kappa_{\partial\Omega}(q) \leq H,$$
so
$$\langle N(q),\eta(q)\rangle^2 \kappa_{\partial\Omega}(q)^2 \leq H^2.$$
Since $\langle N(q),e_3\rangle^2+\langle N(q),\eta(q)\rangle^2=1$, we deduce
$$\langle N(q),e_3\rangle\geq \sqrt{1-\frac{H^2}{\kappa^2_{\partial\Omega}(q)}}=\frac{\sqrt{\kappa^2_{\partial\Omega}(q)-H^2}}{\kappa_{\partial\Omega}(q)}:=C.$$
From \eqref{eq5e-1} and because $H\langle p,e_3\rangle\leq 0$ in $M$, we find
$$\langle N,e_3\rangle\geq C\quad\mbox{in $\Omega$}.$$
Finally, we conclude from \eqref{ifg}
$$|Du|\leq \frac{\sqrt{1-C^2}}{C}\quad\mbox{in $\Omega$}$$
obtaining the desired gradient estimates in $\Omega$. 
\end{proof}

The second proof is done in the two-dimensional case, where the mean convexity is now the convexity in the Euclidean plane. The proof uses spherical caps to find the boundary gradient estimates \eqref{gra}.

\begin{theorem}\label{t5} Let $H\not=0$. If the curvature of $\partial\Omega$ satisfies  $\kappa_{\partial\Omega} \geq |H|$, then the Dirichlet problem \eqref{eq0} has a unique solution.
\end{theorem}

\begin{proof} We start as in the proof of Theorem \ref{t5e} and we follow the same notation. We only need to find  the  {\it a priori}  boundary gradient estimates.  Set
$$\kappa_0=\min_{q\in\partial\Omega}\kappa_{\partial\Omega}(q)>0$$
and $r=1/\kappa_0$. 

Firstly, we prove Theorem \ref{t3} in case of strict inequality $\kappa_{\partial\Omega}>H$. Let $x\in\partial\Omega$ be a fixed but arbitrary boundary point.  Consider $D_r$   a disc of radius $r$ such that $x\in C_r\cap \partial\Omega$ and $\Omega\subset D_r$ where $C_r$ is the boundary of $D_r$. This is possible because $\kappa_0>H$.  Consider   $C_{1/H}$ a circle of radius $1/H$ and concentric to $C_r$. Notice that $r<1/H$. After a translation we suppose that the center of $D_r$ is the origin of coordinates. 

 Let $\s^2(1/H)$ be  the hemisphere   of radius $1/H$ whose boundary is $C_{1/H}$  and  below the plane $\Pi$ of equation $x_3=0$. Let us lift up $\s^2(1/H)$ until its intersection with $\Pi$ is $C_r$. Denote by $S_r$ the piece of $\s^2(1/H)$ below $\Pi$ at this position. See Figure \ref{fig0}. The surface $S_{r}$ is a small spherical cap which is the  graph of 
  $$w^{-}(x_1,x_2)=-\sqrt{\frac{1}{H^2}-x_1^2-x_2^2},\quad x_1^2+x_2^2\leq r^2.$$

We prove now that  $M$ lies in the bounded domain determined by $S_{r}\cup D_r$. For this, we move down $S_r$ by vertical translations until $S_r$ does not intersect $M$ and then, move upwards $S_r$ until the initial position.    Since the mean curvature of $S_r$ is $H$ and $\Omega\subset D_r$, the touching principle implies that there is not a contact before that $S_r$ arrives to its original position. Once we have arrived to the original position,   in a neighborhood of the point $x$, the surface $M$ lies sandwiched between $S_r$ and $\Pi$. Then 
$$Q[w^+]=-2H<0=Q[w^-]=Q[u]$$
and consequently by  Lemma \ref{le1} 
$$\max_{\partial\Omega}|Du|<\max_{\partial S_r}\{|Dw^{-}|,|Dw^{+}|\}=\max_{\partial S_r} |Dw^{-}|=\frac{Hr}{\sqrt{1-H^2r^2}},$$
where this constant depends only on $r$ and $H$.  

\begin{figure}[hbtp]
\centering
\includegraphics[width=7cm]{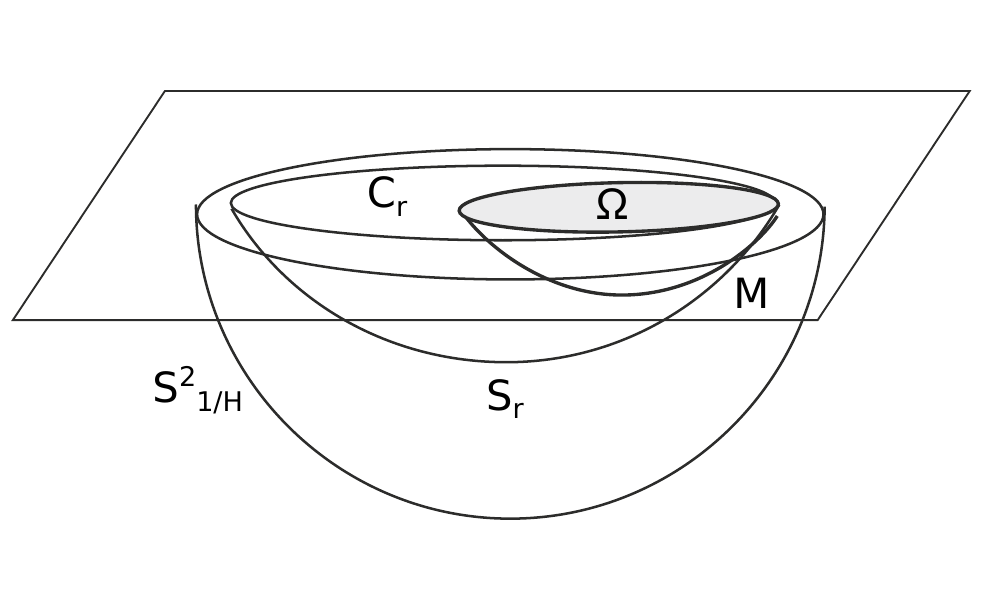}
\caption{Proof of Theorem \ref{t5}.} \label{fig0}
\end{figure}

Until here, we have obtained the existence of a solution for each $0<H<\kappa_0$. Moreover, and since the gradient is bounded from above in $\overline{\Omega}$ depending only on the initial data, the solution obtained is smooth in $\overline{\Omega}$. Now, we proceed by proving the existence of a solution of \eqref{eq1} in the case $H=\kappa_0$: in case that $\Omega$ is a round disk of radius $r$ (and $\kappa_0=1/r$), the solution is   $u(x_1,x_2)=r-\sqrt{r^2-x_1^2-x_2^2}$.

Let us consider an increasing sequence $H_n\rightarrow H$ and $u_n$  the solution of \eqref{eq1} for the value $H_n$ for the mean curvature: the solution exists because  $\kappa_0>H_n$. By the monotonicity of $H_n$ and the comparison principle, the sequence $\{u_n\}$ is monotonically increasing and converges uniformly on compact sets of $\Omega$. Let $u=\lim u_n$. Standard compactness results involving Ascoli-Arzel\'a theorem guarantee that $u\in C^2(\Omega)$ and $Q[u]=0$. It remains to check that $u\in C^0(\overline{\Omega})$ and $u=0$ on $\partial\Omega$. Let $x\in\partial\Omega$ and $\{x_m\}\subset\Omega$ with $x_m\rightarrow x$. Consider the hemisphere $\s^2(r)$  as above and let $D_{r}$ be the open disk of radius $r=1/H$ such that $\s^2(r)=\mbox{graph}(v)$, with $v\in C^\infty(D_{r})\cap C^0(\overline{D_{r}})$. Place $D_r$ such that $x\in\partial D_r$. We know that $\overline{\Omega}\subset D_{r}$ and by the touching principle, $0<u_n<v$ on $\Omega$. For each $n\in\n$, $0<u_n(x_m)<v(x_m)$. Then $0\leq u(x_m)\leq v(x_m)$. Letting $m\rightarrow\infty$, $0\leq u(x)\leq 0$. This proves the continuity of $u$ up to $\partial\Omega$ and that $u=0$ on $\partial\Omega$.
 \end{proof}

%%%%%%%%%%%%%%%%%%%%%%%%%%%%
\section{The Dirichlet problem with zero boundary values: the Lorentzian case}\label{sec8}
%%%%%%%%%%%%%%%%%%%%%%%%%%%%%%%%%%%%%

In this section we address the Dirichlet problem in $\l^3$ following the ideas of the Euclidean case in the above section. The first result that we present is motivated by Theorem \ref{t5}, where we assumed   a strong convexity of $\partial\Omega$ comparing with the value $H$, namely, $\kappa_{\partial\Omega}\geq |H|$.  In contrast, in Lorenz-Minkowski space this   convexity assumption changes by merely the convexity $\kappa_{\partial\Omega}\geq 0$ of $\partial\Omega$.

\begin{theorem}\label{t6} If  $\kappa_{\partial\Omega}\geq 0$,  then the Dirichlet problem 
\begin{equation}\label{eq00}
\left\{\begin{array}{ll}
 \mbox{ div }\Big(\dfrac{D  u}{\sqrt{1- |D  u|^2}}\Big) =2H & \mbox{in $\Omega$}\\
 u=0 & \mbox{on $\partial\Omega$}
\end{array}
\right.\end{equation}
has a unique solution.\end{theorem}

\begin{proof} With a similar argument as in Theorem \ref{t5}, the solutions $u_t$ of the method of continuity are ordered by $u_{t_1}<u_{t_2}$ if $t_2<t_1$, so it suffices to get the  {\it a priori}  estimates for the solution $u$ of \eqref{eq00}. Without loss of generality, we suppose $H>0$. The height estimates were proved in  Theorem \ref{t3} (or \ref{t4}) and we showed  that there exists $K=K(\Omega,H)>0$ such that 
\begin{equation}\label{kk}
-K<u<0\quad\mbox{in $\Omega$}.
\end{equation}
In order to find the   {\it a priori}  boundary gradient estimates, consider the cylinder $C(r)$ determined by $v(x_1,x_2)=\sqrt{r^2+x_1^2}$, where $r=1/(2H)$. For each   $m>r$, let 
$$C(r;m)=\{(x_1,x_2,x_3)\in C_r: x_3\leq m\}.$$
This surface is a graph on the strip $\Omega_{r,m}=\{(x_1,x_2)\in\r^2: -\sqrt{m^2-r^2}\leq x_1\leq \sqrt{m^2-r^2}\}$.  
Take $m$ sufficiently large so $m$ fulfills the next two conditions:
\begin{equation}\label{kk1}
 v(x_1=\sqrt{m^2-r^2})-v(x_1=0)=m-r>K
 \end{equation}
\begin{equation}\label{kk2}
\mbox{diam}(\Omega)<\mbox{width}(\Omega_{r,m})=2\sqrt{m^2-r^2}.
\end{equation}
Let us restrict $v$     in the half-strip 
$$\mathcal{U}=\{(x_1,x_2)\in \Omega_{r,m}:0<x_1<\sqrt{m^2-r^2}\}$$
and $\tilde{C}(r;m)$ denotes   the graph of $v$ on $\mathcal{U}$. The boundary of $\tilde{C}(r;m)$ is formed by two parallel straight-lines 
$$L_1\cup L_2=\{v(x_1=0)\}\cup \{v(x_1=\sqrt{m^2-r^2}\},$$
where $L_1$ is contained in the plane $x_3=r$ and $L_2$ in the plane $x_3=m$, with $r<m$.

Let $x_0\in\partial\Omega$ be a fixed but arbitrary point of the boundary of $\Omega$. After a rotation about a vertical axis and a horizontal translation, we suppose $x_0=(\sqrt{m^2-r^2},0)$, $\Omega$ is contained in $\mathcal{U}$ (this is possible by \eqref{kk2}) and the tangent line $L$ to $\partial\Omega$ at $x_0$ is parallel to the $x_2$-line.  By vertical translations, we displace vertically   down $\tilde{C}(r;m)$ until  it does  not intersect $M=\mbox{graph}(u)$. Then we move vertically upwards   until $ \tilde{C}(r;m)$    intersects $M$ for the first time.
 
 We claim that the first time that  $\tilde{C}(r;m)$ touches $M$ occurs when   $L_2$ arrives to  the plane of equation $x_3=0$ and consequently, $L=L_2$. Firstly, the touching principle prohibits an interior tangent point between $M$ and  $\tilde{C}(r;m)$. On the other hand,   it is not possible that a boundary point of of $C(r;m)$, namely, a point of $L_1\cup L_2$, touches a point of $M$   because \eqref{kk} and \eqref{kk1}.    Definitively, we can move $\tilde{C}(r;m)$ until $L_2$ coincides with $L$, in particular, 
$$x_0 \in L_2 \cap \partial\Omega.$$
At this position, $\tilde{C}(r;m)$ is the graph of the function
$$w^{-}(x_1,x_2)=\sqrt{r^2+x_1^2}-m.$$
Thus  $M$ is contained between    $w^{-}$ and $w^{+}=0$ in $\Omega\cap\mathcal{U}$ with $w^{-}(x_0)=w^{+}(x_0)=u(x_0)=0$. We are in position to apply Lemma \ref{le1} because   $Q[w^{+}]<0=Q[u]=Q[w^{-}]$ and $w^{-}\leq u\leq w^+$ in $\partial(\Omega\cap\mathcal{U})$. We conclude that $|Du|\leq C$, where the constant $C$ in \eqref{gra} is 
$$C=|Dw^{-}|_{|x_1=\sqrt{m^2-r^2}}=\frac{\sqrt{m^2-r^2}}{m}.$$
\end{proof}

The key in the above proof is that the pieces of cylinders $\tilde{C}(r;m)$ of $\l^3$ have arbitrary large height and are graphs on strips of arbitrary width (see \eqref{kk2}). This gives  {\it a priori}  height estimates by choosing $m$ sufficiently large in \eqref{kk2}. Furthermore, the same cylinders provide us the boundary gradient estimates.

With a similar argument, we can derive  {\it a priori}  boundary gradient estimates by using hyperbolic caps. The only difference is that we have to assume strictly convexity $\kappa_{\partial\Omega}>0$.

After Theorem \ref{t6}, we can come back to Euclidean space asking  if it is possible a similar argument by replacing the pieces of cylinders $C(r,m)$ by Euclidean circular  cylinders. Let $H>0$ and consider the  circular cylinder $v(x_1,x_2)=-\sqrt{r^2-x_1^2}$, $r=1/(2H)$ whose mean curvature is $H$ with the orientation given in \eqref{normal}. The only caution is to assure   that the width of any strip containing the (convex) domain $\Omega$ is less than $1/|H|$ as well as its height is less than $1/(2H)$. Again this gives not only the height estimates but also the boundary gradient estimates. With the same ideas as in Theorem \ref{t6}, we prove ( \cite{lo15}):

 \begin{theorem}  Let $H>0$ and $\Omega\subset\r^2$ be a bounded domain with $\kappa_{\partial\Omega}\geq 0$. If 
\begin{equation}\label{l1}
\mbox{dist}(L_1,L_2)<\frac{1}{H},\quad \mbox{for all } (L_1,L_2)\in\mathcal{L},
\end{equation}
 then the Dirichlet problem \eqref{eq0}   has a unique solution. 
\end{theorem}
Comparing this result with Theorem \ref{t5},  the domain here is merely convex even can contain segments of straight-lines; in contrast, the domain $\Omega$ is small in relation to the value of $1/H$. 
\begin{proof} Compare $M=\mbox{graph}(u)$ with the cylinders $C(r)=\mbox{graph}(v)$. An argument as in Theorem \ref{t6} proved that the hypothesis \eqref{l1} ensures  that $-1/(2H)<u<0$ in $\Omega$: in fact, for this estimate it suffices that \eqref{l1} holds for {\it one} pair of lines $(L_1,L_2)\in\mathcal{L}$. The boundary gradient estimates follow comparing with quarter of cylinders $C(r)$ defined in the strip $0\leq x_1\leq 1/(2H)$. 
\end{proof}

The following result solves affirmatively the Dirichlet problem in the Lorentz-Minkowski space \eqref{eq00}  for {\it arbitrary } domains. For this,  we will use cmc rotational spacelike  surfaces of $\l^3$ as barriers. We now describe  the rotationally symmetric solutions of \eqref{eq1}. 
 
 Consider a rotational surface about the $x_3$-axis obtained by the curve $(r,0,w(r))$, $0\leq a<r<b$. With respect to the orientation   \eqref{normal}, the mean curvature $H$ satisfies
\begin{equation}\label{lcmcgraph} 
 \frac{w''}{(1-w'^2)^{3/2}}+\frac{w'}{r\sqrt{1-w'^2}}=2H.
\end{equation}
The spacelike condition is equivalent to $w'^2<1$. Multiplying by $r$,  a first integral is
$$Hr^2+c=\frac{r w'}{\sqrt{1-w'^2}}$$
for a constant $c\in\r$, or equivalently 
\begin{equation}\label{wl1}
w'=\pm \frac{Hr^2+c}{\sqrt{r^2+(Hr^2+c)^2}}\cdot
\end{equation}
 If $c=0$, the solution is $w(r)=\sqrt{1/H^2+r^2}$, up to a constant, that corresponds with a  hyperbolic plane $\h^2(1/H)$.

Let $H>0$ and $c<0$. Since $w'^2<1$, the function  $w$ is defined in $(0,\infty)$. By \eqref{wl1}, $w''>0$ and  $w'$ vanishes at a unique point, namely, $r_0=\sqrt{-c/H}$.   It is also clear that 
$\lim_{r\rightarrow 0}w'(r)=-1$. Consider $w=w(r;c)$ be the solution of \eqref{wl1} parametrized by the constant $c$ assuming  initial condition 
\begin{equation}\label{wl2}
w(r_0)=0,\quad (\mbox{so } w'(r_0)=0).
\end{equation}
Let $S(c)$ denote the graph of $w(r;c)$ with $r^2=x_1^2+x_2^2$.
  See Fig. \ref{fig1}, left.  Let $\xi_c=\lim_{r\rightarrow 0}w(r;c)$. The functions $w(r;c)$ have the following properties.  
\begin{enumerate}
\item $S(c)$  presents a singularity at the intersection point with the rotation axis. See Fig. \ref{fig1}, right. At this point, the surface is tangent to the (backward) light-cone from $w(0;c)$, namely, 
$$x_1^2+x_2^2=(x_3-\xi_c)^2,\quad x_3<\xi_c.$$
\item $\lim_{c\rightarrow-\infty}r_0(c)=+\infty$ and $\lim_{c\rightarrow-\infty}\xi_c=+\infty$.

\item $\lim_{c\rightarrow 0}r_0(c)=0$ and  $\lim_{c\rightarrow 0}\xi_c=0$.
\end{enumerate}

\begin{figure}[hbtp]
\centering
\includegraphics[width=5cm]{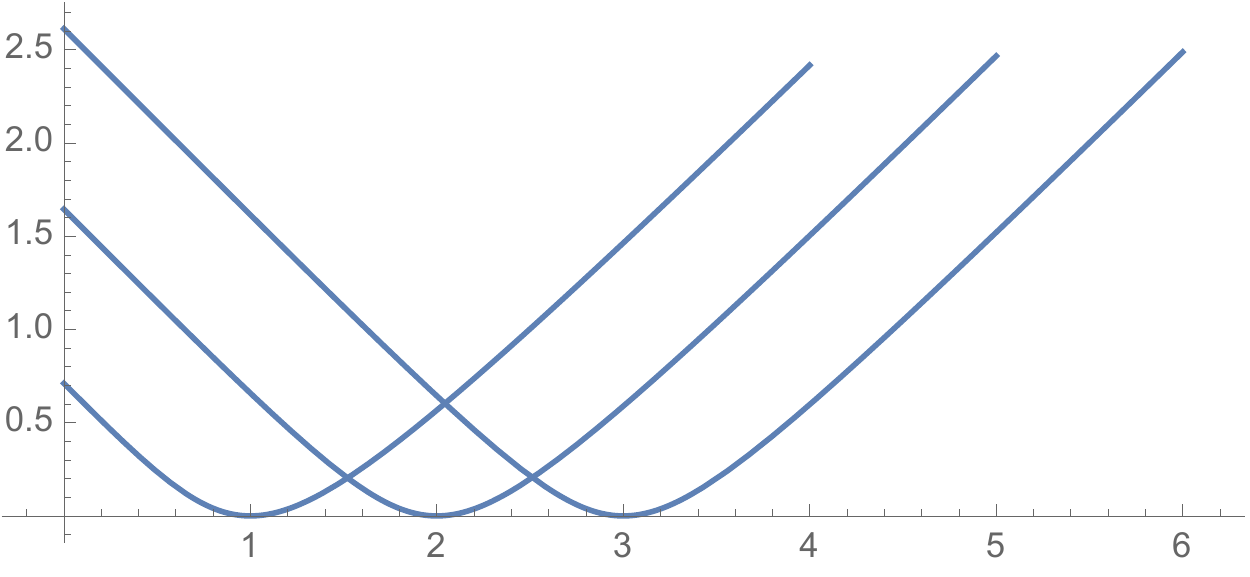}\quad\includegraphics[width=7cm]{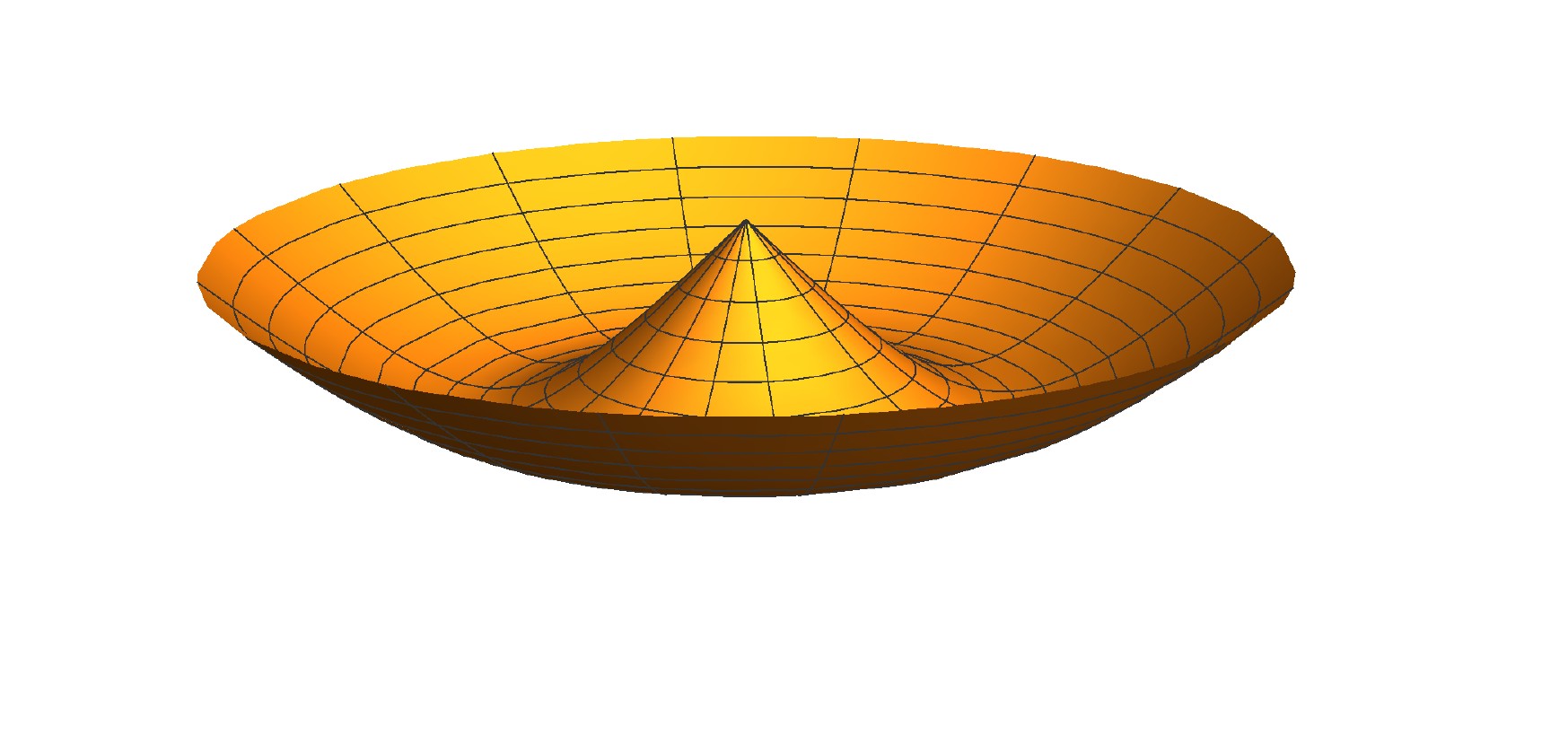}
\caption{Left:   profiles of generating curves of cmc rotational   spacelike  surfaces $S(c)$ for values $c=1$, $c=2$ and $c=3$. Right: a cmc rotational  spacelike surface.} \label{fig1}
\end{figure}
 
 The following result has not a counterpart in the Euclidean space.
 
 \begin{theorem}\label{t7} 
If $\Omega$ is a bounded smooth domain, then the Dirichlet problem \eqref{eq00}   has a unique solution. 
 \end{theorem} 
\begin{proof}

If $H=0$, the solution is the function $u=0$. Let $H\not=0$. By changing $u$ by $-u$ if necessary, without loss of generality we suppose   that $H>0$. 
 We know by Theorem \ref{t3} that $u<0$ in $\Omega$. As in Theorem \ref{t6},   it suffices to find  {\it a priori}  estimates for the solution $u$ of (\ref{eq1}) which corresponds with the value $t=1$. Moreover,  the function $w^{+}=0$ is an upper barrier because $Q[w^{+}]=-2H<0$ in $\Omega$ and $w^{+}=u$ along $\partial \Omega$. In order to find lower barriers for $u$, we will take pieces of cmc rotational  surfaces $S(c)$ for suitable choices of the parameter $c$   depending only on the initial data.

 Since $\Omega$ is smooth ($C^2$ is enough), $\Omega$ satisfies a uniform exterior circle condition. This means that there exists a small enough $\varepsilon>0$ depending only on $\Omega$ with the following property: for any boundary point $x\in\partial\Omega$,   there is a disc $D_\varepsilon$ of radius $\varepsilon$ and depending on $x$ such that 
 $$D_\varepsilon\cap\Omega=\emptyset,\quad \overline{D_\varepsilon}\cap\overline{\Omega}=\{x\}.$$
As  consequence, the same property   holds for every $\varepsilon'>0$ with $\varepsilon'\leq\varepsilon$.
 
Fix the above $\varepsilon$. Let $w=w(r;c)$ be a solution of \eqref{wl1}-\eqref{wl2}   defined only in the interval $[\varepsilon,r_0]$ and let $S(c;\varepsilon)$ be its graph. Here, and in what follows, we identify the function $w=w(r)$ of one variable with the rotationally symmetric function of two variable $w=w(x_1,x_2)$ by setting $x_1^2+x_2^2=r^2$. Then the boundary of $S(c;\varepsilon)$ are the   circles 
$$\partial S(c;\varepsilon)=C_1\cup C_2:=\{(x_1,x_2,w(\varepsilon;c)):x_1^2+x_2^2=\varepsilon^2\}\cup \{(x_1,x_2,0):x_1^2+x_2^2=r_0^2\}.$$
By the height estimates of Theorem \ref{t3}, there exists a constant $K>0$ depending only on the initial data such that $-K<u<0$ in $\Omega$. Let $c<0$ be sufficiently small with the  next two   properties
\begin{equation}\label{r00}
r_0(c) >\mbox{diam}(\Omega),\quad w(\varepsilon;c)>K.
\end{equation}
 Given $\varepsilon$, the last inequality is a consequence of $\xi_c\rightarrow\infty$ as $r_0\rightarrow-\infty$. Let $w^{-}=w(r;c)$.  
 
 Let $x\in\partial\Omega$ be a boundary point and let $D_\varepsilon$ be the disc given by the uniform exterior circle condition. We now prove that it is possible to choose a suitable $S(c;\varepsilon)$   such that $S(c;\varepsilon)$ is a lower barrier for $u$ around the point $x$. In what follows, we denote by the same symbol $S(c;\varepsilon)$ any vertical translation of this surface which corresponds with the functions $w(r;c)+k$ for different choices of the constant $k$.

 After a horizontal translation, we suppose  
$x=(\varepsilon,0)$ and that the disc $D_\varepsilon$ of the uniform exterior circle condition is $x_1^2+x_2^2<\varepsilon^2$. We move vertically down the surface $S(c;\varepsilon)$ until that it does not intersect $M=\mbox{graph}(u)$. Then we come back by lifting vertically upwards $S(c;\varepsilon)$. 

\begin{figure}[hbtp]
\centering
\includegraphics[width=10cm]{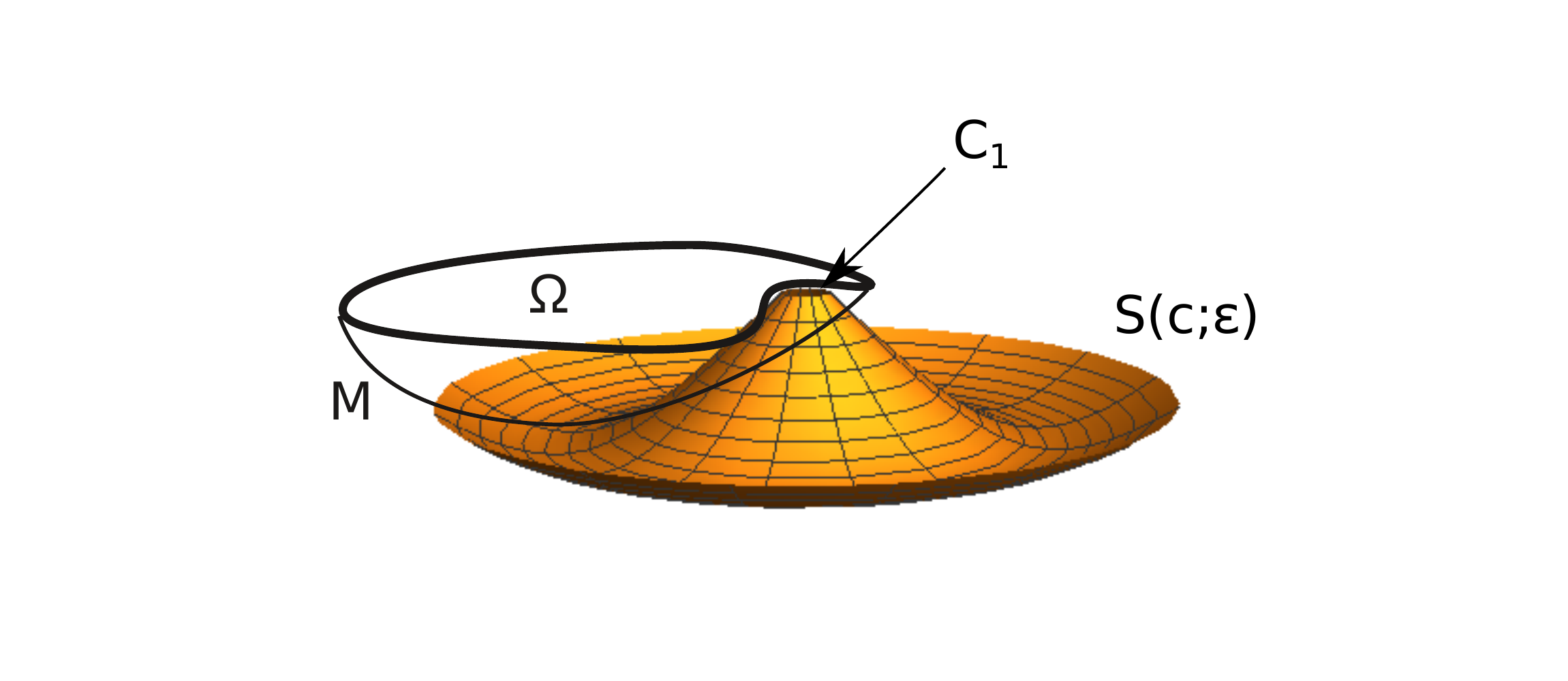}
\caption{The surface $S(c;\varepsilon)$ is a lower barrier for the graph $M$.}\label{fig2}
\end{figure}

{\it Claim.} It is possible to move upwards $S(c;\varepsilon)$ without touching $M$ until  that we place $S(c;\varepsilon)$ just at the position where  the boundary circle $ C_1$ coincides with   $\partial D_\varepsilon$.   See Fig. \ref{fig2}.

This occurs because the touching principle forbids a first contact at some common interior point. The other possibility is that during the vertical displacement, and before to arrive to the final position, some boundary point of $S(c;\varepsilon)$, namely, a point of $C_2$, touches $M$: the  circle $ C_1$ does not touch $M$ because $D_\varepsilon\cap\Omega=\emptyset$. The other circle $ C_2$ projects onto $\r^2$ in the circle $x_1^2+x_2^2=r_0^2$ which contains $\Omega$ inside   by the first property of \eqref{r00}. Finally, the circle $C_2$ does not touch $M$ because the vertical distance between $C_1$ and $C_2$ is $w^{-}(\varepsilon;c)-w^{-}(0;c)=w(\varepsilon;c)>K$ by \eqref{r00}.

Once we have placed $S(c;\varepsilon)$ so that $ C_1=\partial D_\varepsilon$, the lower barrier is $w^{-}=w(r;c)-w(\varepsilon;c)$ defined in the annulus $\mathcal{U}=\{(x_1,x_2): \varepsilon^2<x_1^2+x_2^2<r_0^2\}$. We deduce that $w^-<u$ in $\Omega\cap\mathcal{U}$. This proves that $|Du(x)|<|Dw^-(x)|$ by Lemma \ref{le1} and  this value depends only on the initial data, namely,  
$$|Dw^-(x)|=-\frac{d}{dr}{\Big|}_{r=\varepsilon}w(r;c)=-\frac{H\varepsilon^2+c}{\sqrt{\varepsilon^2+(H\varepsilon^2+c)^2}}\cdot.$$
This gives the constant $C$ in \eqref{gra}.
\end{proof}

 %%%%%%%%%%%%%%%%%%%%%%%%%%%%%%%%%%%%%%%%%%
%\funding{ ``'' }

%%%%%%%%%%%%%%%%%%%%%%%%%%%%%%%%%%%%%%%%%%
   
%%%%%%%%%%%%%%%%%%%%%%%%%%%%%%%%%%%%%%%%%%
 

%\externalbibliography{yes}
%\bibliography{your_external_BibTeX_file}

%=====================================
% References, variant B: internal bibliography
%=====================================

%%%%%%%%%%%%%%%%%%%%%%%%%%%%%%%%%%%%%%%%%%
\end{document}